\documentclass[12pt,reqno]{amsart}
\usepackage[margin=1in]{geometry}
\usepackage[normalem]{ulem}

\usepackage{amsmath,amsthm,amssymb,color,latexsym}

\usepackage{graphicx}
\usepackage[shortlabels]{enumitem}

\usepackage{multirow}

\usepackage{hyperref}

\theoremstyle{plain}

\newtheorem{thm}{Theorem}

\newtheorem{lem}{Lemma}
\theoremstyle{definition}

\theoremstyle{remark}

\numberwithin{equation}{section}

\newcommand\rI{\mathrm{I}}

\newcommand\C{\mathbb{C}}
\newcommand\R{\mathbb{R}}

\newcommand\vc{v_c}

\def\tht{\theta}
\def\Om{\Omega}

\def\e{\varepsilon}

\def\G{\Gamma}
\def\l{\lambda}
\def\p{\partial}

\def\b{\beta}
\def\Ups{\Upsilon}
\def\k{\kappa}

\def\d{\delta}
\def\L{\Lambda}

\def\iu{\mathrm{i}}

\def\Op{\mathcal{H}}

\def\rA{\mathrm{A}}
\def\rQ{\mathrm{Q}}
\def\rE{\mathrm{E}}

\def\cR{\mathcal{R}}
\def\cP{\mathcal{P}}

\def\cB{\mathcal{B}}
\def\cL{\mathcal{L}}

\def\cS{\mathcal{S}}
\def\cI{\mathcal{I}}

\def\cC{\mathcal{C}}

\def\rE{\mathrm{E}}

\def\rM{\mathrm{M}}
\def\rT{\mathrm{T}}
\def\rP{\mathrm{P}}
\def\rU{\mathrm{U}}
\def\rV{\mathrm{V}}
\def\rA{\mathrm{A}}
\def\rB{\mathrm{B}}
\def\rD{\mathrm{D}}

\DeclareMathOperator{\RE}{Re}
\DeclareMathOperator{\IM}{Im}

\DeclareMathOperator{\Dom}{\mathfrak{D}}

\DeclareMathOperator{\rank}{rank}

\DeclareMathOperator{\Ker}{Ker}

\renewcommand{\geq}{\geqslant}

\begin{document}

\title[Eigenvalue and resolvent asymptotics for a graph with a
shrinking edge]{Exotic eigenvalues and analytic resolvent for a graph with a shrinking edge}

\author{Gregory Berkolaiko}
\address{Department of Mathematics, Texas A\&M University, College Station, TX 77843-3368, USA}

\author{Denis I. Borisov}
\address{Institute of Mathematics, Ufa Federal Research Center,
  Russian Academy of Sciences,  Chernyshevsky str. 112, Ufa, Russia,
  450008 \quad and \quad
University of Hradec Kr\'alov\'e, Rokitansk\'eho 62, Hradec
  Kr\'alov\'e 50003, Czech Republic}

\author{Marshall King}
\address{Department of Mathematics, Texas A\&M University, College Station, TX 77843-3368, USA}




\begin{abstract}
  We consider a metric graph consisting of two edges, one of which has
  length $\e$ which we send to zero.  On this graph we study the
  resolvent and spectrum of the Laplacian subject to a general vertex
  condition at the connecting vertex.  Despite the singular nature of
  the perturbation (by a short edge), we find that the resolvent
  depends analytically on the parameter $\e$.  In contrast, the
  negative eigenvalues escape to minus infinity at rates that could be
  fractional, namely, $\e^0$, $\e^{-2/3}$ or $\e^{-1}$.  These rates
  take place when the corresponding eigenfunction localizes,
  respectively, only on the long edge, on both edges, or only on the
  short edge.
\end{abstract}

\maketitle

\section{Introduction}

Differential operators on metric graphs arise in numerous applied
problems, for example, as effective descriptions of physical processes
taking place on thin branching domains
\cite{KucZen_incol03,Gri_incol08,ExnPos_jpa09,Post_book12}.  Spectral
properties of such operators depend on many factors, such as the
differential expression itself, vertex matching conditions,
connectivity (topology) of the graph, as well as edge lengths.  In this
study we focus on a graph which consists of edges of two length
scales, of order one and of order $\e \to 0$.  Such problems
arise naturally in the studies of metamaterials, where a large-scale
structure may contain small-scale inclusions which substantially
alters the overall physical properties \cite{CheExnTur_ap10,
  DoKucOng_ems17,LawTanChr_sr22}.

While analytic dependence of a compact graph's spectrum on the edge
lengths was known for some time \cite{BerKuc_incol12}, this result
specifically excluded the case of edges shrinking to a point.
Substantial progress was achieved in four recent publications
\cite{BerLatSuk_am19,Cac_s19,Bor_am22,Bor_m21}, where general positive results were
established under varying ``non-resonance'' conditions, which,
informally speaking, prevent eigenfunctions from localizing on the
shrinking part of the graph.  In this work we thoroughly study a simple example
that violates these conditions.

Despite the simplicity of the example, we catalogue a variety of
curious behaviors.  To give a preview, consider the operator acting as
$-\frac{d^2}{dx^2}$ on the space $L^2((-\e, 0))\oplus
L^2((0,1))$, supplied with the following vertex conditions
\begin{equation}
  \label{eq:preview_example}
  u'(-\e) = 0,
  \qquad
  u'(0-) = u(0+),
  \qquad
  u'(0+)=-u(0-),
  \qquad
  u(1)=0.
\end{equation}
An a priori bound by Kuchment \cite{Kuc_wrm04} (see also
\cite{KosSch_incol06,BolEnd_ahp09}) estimates the bottom of the
spectrum to be $\gtrsim -1/{s}_{\min}$, where ${s}_{\min}$ is the
shortest edge length, i.e. $\e$.  Surprisingly, in this
particular example, the lowest eigenvalue tends to $-\infty$ at a
fractional rate, namely $\lambda_1 = - \e^{-2/3} + O(\e^{2/3})$.

It is interesting to compare this example to the problem of absorption of
eigenvalues into the continuous spectrum, studied by Simon in \cite{Sim_jfa77}.
Rescaling all edge lengths by $\e$ and extending the long edge to
infinity, we arrive at the eigenvalue problem for the
Laplacian on $(-1,0) \cup (0,\infty)$ with vertex conditions
\begin{equation}
  \label{eq:preview_example_inf}
  u'(-1) = 0,
  \qquad
  u'(0-) = \e u(0+),
  \qquad
  u'(0+)=-\e u(0-).
\end{equation}
Here, the lone negative eigenvalue approaches the continuous spectrum
at $[0,\infty)$ at the rate $\e^{4/3}$.  In \cite[end of
Sec.~2]{Sim_jfa77}, Simon argued that one can obtain any rate
$\e^\alpha$, $\alpha \geq 1$, by considering a fractional power of
Laplacian.  Here we obtain a fractional rate for the Laplacian itself
and only with linear dependence of the vertex conditions on the
parameter $\e$.

In this work we take this two-edge
graph and search through all possible vertex conditions at the
connecting vertex, in order to classify all possible rates attainable
by the negative eigenvalues and, through a detailed study of the resolvent
and the eigenfunctions, understand the circumstances in which the
fractional rates arise.  Interestingly, we find that the leading order
rates can be $\e^0$, $\e^{-2/3}$ or $\e^{-1}$ and \emph{nothing else}.
This encourages us to conjecture that the same holds for arbitrary
graphs with edge lengths of two scales.

\section{Problem setting and the main results}
\label{sec:results}

We consider the graph $\Gamma_\varepsilon$ consisting of an edge
${s}_\e$ of length $\e$ connected to an edge $e$ of length $1$, see
Figure \ref{fig:example}; here $\e$ is a small positive parameter.
The internal vertex connecting the two edges is denoted by $\vc$,
while the other two vertices being the end-points of the edges ${s}_\e$
and $e$ are respectively denoted by $v_\e$ and $v$. On the edges we
introduce variables, which are respectively denoted by $x_\e$ and
$x$. The orientation is fixed by letting $x_\e$ range from 0 at
$v_\e$ to $\e$ at $\vc$ and $x$ range from 0 at $v_1$ to $1$ at $\vc$.

\begin{figure}[t]
  \centering
  \includegraphics[width=8cm]{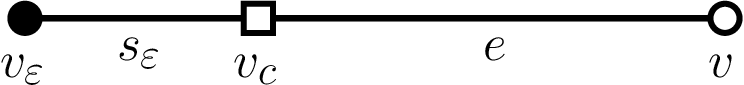}
  \caption{The graph $\Gamma_\varepsilon$. The vertex conditions are Neumann at $v_\e$, Dirichlet at $v$, and arbitrary at $v_c$.}
  \label{fig:example}
\end{figure}

We consider the self-adjoint operator $\Op_\e$ on
$L_2(\G_\e):=L_2({s}_\e)\oplus L_2(e)$, acting as
\begin{equation}\label{2.1}
  \Op_\e \colon
  (u_\e, u) \mapsto
  \left(-\frac{d^2u_\e}{dx_\e^2}, -\frac{d^2u}{dx^2}\right)
\end{equation}
with the domain $\Dom(\Op_\e)$  consisting of the functions $(u_\e,u)\in H^2({s}_\e)\oplus H^2(e)$ satisfying the boundary conditions
\begin{equation}\label{2.2}
u_\e'(0)=0,\qquad u(0)=0, \qquad \rP \rU=0,\qquad \rQ \rU'=\rT \rQ\rU,
\end{equation}
where
\begin{equation}\label{2.3}
\rU:= \begin{pmatrix}
u_\e(\e)
\\
u(1)
\end{pmatrix},\qquad \rU':= \begin{pmatrix}
-u_\e'(\e)
\\
-u'(1)
\end{pmatrix},
\end{equation}
$\rP$ is an arbitrary orthogonal projection operator acting in $\C^2$,
and $\rT$ is an arbitrary self-adjoint operator acting on the range of
$Q:=\rI-\rP$.  By Theorem 1.4.4 of \cite{BerKuc_graphs}, the last two
conditions in (\ref{2.2}) represent arbitrary\footnote{We also
  considered other descriptions of the vertex conditions, such as
  those listed in \cite[Thm.~1.4.4]{BerKuc_graphs} as well as the
  parametrization introduced in \cite{ExnGro_prep99}.  The
  parametrization we use in \eqref{2.2} results in the least
  cumbersome classification of the asymptotic behaviors,
  Table~\ref{tab:neg_eig}.} self-adjoint conditions at the vertex
$\vc$ and the introduced operator is self-adjoint.

Since the operator $\Op_\e$ is defined on a compact graph, its
resolvent is a compact operator in $L_2(\G_\e)$ and its spectrum
consists of discrete eigenvalues, which can accumulate at infinity
only.  The main aim of this work is to study the behavior of the
resolvent and eigenvalues of the operator $\Op_\e$ as $\e$
goes to zero.  Specifically, we focus on the negative eigenvalues ---
and the corresponding eigenfunctions --- as their behavior is most
strongly affected by the ``singular'' limit $\e\to0$.

We further parametrize vertex conditions~\eqref{2.2}
by considering the following three cases.
\begin{enumerate}
\item
If $\rank\rP=2$, then
\begin{equation}\label{2.4}
  \rP=\mathrm{I},\qquad \rQ=0,\qquad\rT=0.
\end{equation}
This case corresponds to the (decoupled) Dirichlet conditions at the central vertex.
\item
If $\rank\rP=1$, then the matrices
$\rP$ and $\rQ$ can be represented as
\begin{equation}\label{eq:parametrization}
\begin{aligned}
&\rP = \begin{pmatrix}
\frac{1}{\sqrt{1+|z|^2}}
\\
\frac{\overline{z}}{\sqrt{1+|z|^2}}
\end{pmatrix}
\begin{pmatrix}
\frac{1}{\sqrt{1+|z|^2}}
 &
\frac{z}{\sqrt{1+|z|^2}}
\end{pmatrix}=\begin{pmatrix}
                \frac{1}{1+|z|^2} & \frac{z}{1+|z|^2} \\
                \frac{\bar z}{1+|z|^2} & \frac{|z|^2}{1+|z|^2}
              \end{pmatrix},
              \qquad z \in \C \cup \{\infty\},
\\
&\rQ = \begin{pmatrix}
  \frac{|z|^2}{1+|z|^2} & -\frac{z}{1+|z|^2} \\
  -\frac{\bar z}{1+|z|^2} & \frac{1}{1+|z|^2}
\end{pmatrix},
\end{aligned}
\end{equation}
and $T$ acts as a multiplication by $\mu\in\R$.  This case includes
Neumann-Kirchhoff ($z=-1$, $\mu=0$) and delta-type conditions ($z=-1$,
$\mu\neq0$).  When $z=0$ or $z=\infty$, the central vertex decouples
into one Dirichlet and one Neumann condition.
\item If $\rank\rP=0$, then
\begin{equation}\label{2.5}
\rP=0,\qquad \rQ=\mathrm{I}, \qquad \rT=  \begin{pmatrix}
          a & c \\
            \bar c & b
        \end{pmatrix}
\end{equation}
with some $a,b\in\R$, $c\in \C$. This case corresponds to a
generalized Robin condition $\rU'=\rT \rU$, which is decoupled if
$\rT$ is diagonal.
\end{enumerate}
Our main results are as follows.

\begin{thm}
  \label{tm:asym}
  The operator $\Op_\e$ can have at most two negative eigenvalues.
  Each eigenvalue is simple and comes in one of the three possible types
  (in the description below, $\L_0$ is a function holomorphic at 0 and
  satisfying $\L_0(0)\neq 0$, $\L_1$ is a smooth function satisfying
  the estimate $\L_1(s) = O\big(e^{-\frac{C}{s}}\big)$ for small
  positive $s$ and a positive constant $C$, and $\alpha>0$ is a real constant):
  \begin{enumerate}
  \item[(B)] a bounded eigenvalue,
    \begin{equation}
      \label{eq:eigB}
      \lambda(\e) = -\L_0^2(\e) = -\alpha + O(\e),
    \end{equation}
  \item[(S)] an eigenvalue depending on the square root $\e^{\frac{1}{2}}$,
    \begin{equation}
      \label{2.20}
      \l(\e)=-\e^{-1}
      \left(\L_0(\e^\frac{1}{2})+\L_1(\e^\frac{1}{2})\right)^2
      = -\alpha \e^{-1} + O(\e^{-\frac{1}{2}}),
    \end{equation}
  \item[(C)] an eigenvalue depending on the cubic root $\e^{\frac{1}{3}}$,
    \begin{equation}
      \label{2.21}
      \l(\e)=-\e^{-\frac{2}{3}}\left(\L_0(\e^\frac{1}{3})+\L_1(\e^\frac{1}{3})\right)^2
      = -\alpha \e^{-\frac{2}{3}} + O(\e^{-\frac{1}{3}}).
    \end{equation}
  \end{enumerate}
  The negative eigenvalues exist only in the cases specified in
  Table~\ref{tab:neg_eig}, where $\kappa_1$ is the root of
  \begin{equation}
    \label{eq:k1_def}
    \kappa \coth\kappa = -\mu(1+|z|^2),
  \end{equation}
  and $\kappa_0$ is the root of
  \begin{equation}
    \label{eq:k0_def}
    \kappa \coth\kappa = \frac{|c|^2 - ab}{a}, \quad a\neq 0,
    \qquad\text{or}\qquad
    \kappa \coth\kappa = -b, \quad a=0,\quad c=0.
  \end{equation}
\end{thm}

\renewcommand{\arraystretch}{1.2}
\setlength{\tabcolsep}{8pt}
\begin{table}
  \centering
  \begin{tabular}{c|c|c||c|c|c}
    \multicolumn{3}{c||}{Conditions}&type(B)&type(S)&type(C)\\[5pt]
    \hline
    \multirow{2}{*}{$\rank P=1$} & $z = \infty$ & $\mu<0$
                                   & & \checkmark $\alpha=|\mu|$ & \\
    \cline{2-6} & $z \neq \infty$ & $\mu(1+|z|^2)<-1$
                                   & \checkmark $\alpha=\kappa_1^2$ & & \\
    \hline \multirow{5}{*}{$\rank P = 0$}
                                   & \multirow{2}{*}{$a<0$} & $|c|^2 < a(b+1)$
                                 & \checkmark $\alpha=\kappa_0^2$ & \checkmark $\alpha=|a|$ & \\
    \cline{3-6} & & $|c|^2 \geq a(b+1)$
                                 & & \checkmark $\alpha=|a|$& \\
    \cline{2-6}
                                   &\multirow{2}{*}{$a=0$}
                                           & $c=0$,\ $b+1<0$
                                   & \checkmark $\alpha=\kappa_0^2$ & & \\
    \cline{3-6} & & $c\neq 0$
                                   & & & \checkmark $\alpha=|c|^{4/3}$ \\
    \cline{2-6} & $a>0$ & $|c|^2 > a(b+1)$
                                   & \checkmark $\alpha=\kappa_0^2$ & & \\ \hline
  \end{tabular}
  \caption{Description of the vertex condition which lead to negative
    eigenvalues.  The coefficient $\alpha$ of the leading term in
    \eqref{eq:eigB}-\eqref{2.21} is specified in the corresponding
    column.  When no coefficient $\alpha$ is specified, there is no
    eigenvalue of the corresponding type.  If the vertex conditions do
    not fit into either of the listed cases, there are no negative
    eigenvalues.  For explanations and examples of the vertex
    conditions, see \eqref{2.4}-\eqref{2.5}.}
  \label{tab:neg_eig}
\end{table}

\begin{thm}
  \label{th:efs}
  As $\e\to0+$, the normalized eigenfunctions of the operator $\Op_\e$
  associated with its negative eigenvalues satisfy
  \begin{align}\label{2.23}
    &\big\|\psi|_{s_\e}\big\|^2=O(\e),
    && \big\|\psi|_e\big\|^2=1+O(\e)
    && \text{for eigenvalue of type (B)},
    \\
    &\big\|\psi|_{s_\e}\big\|^2=1+O(\e^{1/2}),
    && \big\|\psi|_e\big\|^2=O(\e^{1/2})
    &&\text{for eigenvalue of type (S)},
       \label{2.24}
    \\
    &\big\|\psi|_{s_\e}\big\|^2=\frac{2}{3}+O(\e^{1/3}),
    && \big\|\psi|_e\big\|^2=\frac{1}{3}+O(\e^{1/3})
    &&\text{for eigenvalue of type (C)}.
       \label{2.25}
  \end{align} 
\end{thm}

It is interesting to note that in cases (S) and (C) a
non-vanishing proportion of the eigenfunction's norm localizes on the
edge of vanishing length.

To fully describe the resolvent we need to introduce further notation.
We let ${s}:=(0,1)$ and introduce the bounded operator
$\cS_\e: L_2({s}) \to L_2({s}_\e)$ acting as
$(\cS_\e u)(x):=u(\frac{x}{\e})$.  The mapping $\cS_\e \oplus \cI$ is
an isomorphism between linear spaces $L_2({s})\oplus L_2(e)$ and
$L_2(\G_\e) = L_2({s}_\e)\oplus L_2(e)$; we stress that this isomorphism does not preserve the
  $L_2$-norm. By $\cP_s: L_2(\G_\e)\to L_2({s}_\e)$ and
$\cP_e: L_2(\G_\e)\to L_2(e)$ we denote the natural restriction operators
\begin{equation}
  \label{eq:restriction_def}
  \cP_s: (u_\e, u) \mapsto u_\e,
  \qquad
  \cP_e: (u_\e,u) \mapsto u.
\end{equation}

Since the operator $\Op_\e$ is self-adjoint, its resolvent
$(\Op_\e-\l)^{-1}$ is well-defined for $\l\in\C\setminus\R$. We
introduce two auxiliary operators on the space
$L_2({s})\oplus L_2(e)$ by the formulas
\begin{equation*}
 \cR_e(\e,\l):=\cP_e (\Op_\e-\l)^{-1} (\cS_\e\oplus\cI),\qquad
 \cR_{{s}}(\e,\l):=\cS_\e^{-1}\cP_s (\Op_\e-\l)^{-1} (\cS_\e\oplus\cI).
\end{equation*}
Let us clarify the action of these operators. Given an element
$f=(f_{s},f_e)\in L_2({s})\oplus L_2(e)$, we let
$F=(\cS_\e\oplus \cI)f=(\cS_\e f_{s},f_e)$, which is a function in
$L_2(\G_\e)$.  We then apply the resolvent $(\Op_\e-\l)^{-1}$ to $F$
and the restriction of the result to $e$ is the action of the operator
$\cR_e$, while the restriction to the small edge ${s}_\e$ rescaled by
$\cS_\e^{-1}$ is the action of the operator $\cR_{{s}}$. It is clear
that $\cR_e$ and $\cR_{s}$ are bounded operators from
$L_2({s})\oplus L_2(e)$ into $L_2(e)$ and $L_2({s})$. Expressing then the second derivatives of $\cR_{s} (f_{s},f_\e)$ and $\cR_e (f_{s},f_\e)$ from the corresponding equations, see (\ref{3.1}), we see immediately that   the operators $\cR_e$ and $\cR_{s}$ are also bounded as acting into $H^2(e)$ and $H^2({s})$. We stress that here we mean just boundedness of the operators $\cR_e$ and $\cR_{s}$ and not a uniform boundedness of their norms in $\e$.
They can be
regarded as parts of the resolvent $(\Op_\e-\l)^{-1}$ in the sense of
the obvious identity
\begin{equation}\label{2.9}
(\Op_\e-\l)^{-1}= \big(\cS_\e\cR_{s}(\e,\l)\oplus\cR_e(\e,\l)\big) (\cS_\e^{-1}\oplus \cI).
\end{equation}
To state our results we also introduce an auxiliary operator
$\cL: L_2(e)\to H^2(e)$,
\begin{equation}
  \label{eq:Ldef}
  (\cL f_e)(x):=-\frac{1}{\sqrt{\l}}\int\limits_{0}^{x} \sin\sqrt{\l}(x-t) f_e(t)\,dt.
\end{equation}

\begin{thm}\label{th:res}
  For a fixed $\l\in \C\setminus\R$ the operators $\cR_{s}$ and
  $\cR_e$ are holomorphic in $\e$ as operators from
  $L_2({s})\oplus L_2(e)$ into $L_2({s})$ and $L_2(e)$
  correspondingly.

  For all choices of vertex conditions at
  $\vc$, the leading order of $\cR_e$ is given by
  \begin{equation}
    \label{eq:Re}
    \cR_e(\e,\l)f = \cB f_e \sin\sqrt{\l}x +\cL f_e+ O(\e),
  \end{equation}
  where $\cB: L_2(e)\to\C$ is a bounded linear functional and $\cL$ is
  given by~\eqref{eq:Ldef}.

  The leading terms of the Taylor series for the operator $\cR_{s}$
  are
  \begin{equation}
    \label{eq:Rell}
    \cR_{s}(\e,\l)f =
    \begin{cases}
      O(\e^2), & \text{if } \rank\rP=2, \\
      \e\cB_{s} f_{s} + O(\e^2),
      &\text{if }\rank\rP=1,\ z=\infty,\ \mu\ne0,\\
      \cB_{s} f_{s} +O(\e),
      &\text{if }\rank\rP=1,\ z=\infty,\ \mu=0,\\
      &\text{ or }\rank\rP=0,\ a=c=0 \\
      \cB_e f_e+O(\e) & \text{if }\rank\rP=1,\ z\neq \infty,\\
      &\text{ or } \rank\rP=0,\ (a,c)\neq(0,0),
    \end{cases}
  \end{equation}
  where $\cB_e$ and $\cB_{s}$ are some bounded linear functionals on
  $L_2(e)$ and on $L_2({s})$, correspondingly.
\end{thm}

It is interesting to note that both parts of the resolvent are
holomorphic despite that in some cases, detailed in Theorem~\ref{tm:asym},
the eigenvalues are functions of fractional powers of $\e$.

\subsection{Comparison with previous results}
\label{sec:discussion}

Let us briefly discuss our results in comparison to previous related
works \cite{BerLatSuk_am19,Bor_am22} (see also \cite{Cac_s19} which
has a different scaling in the vertex conditions).

\newcommand\Oplim{\Op_{\mathrm{lim}}}

The focus of \cite{BerLatSuk_am19} was on the norm resolvent
convergence to the natural limiting graph operator $\Oplim$ obtained
from \eqref{2.1}-\eqref{2.2} as follows: only the length 1 edge
remains and the vertex condition at $\vc$ is obtained by
substituting\footnote{Intuitively, the derivative $u'_\e$ does not
  change very much over a short edge; since $u'_\e(0)=0$, we also
  expect $u'_\e$ to be close to 0 on the other end.} $u'_\e(\e) = 0$
and eliminating $u_\e(\e)$ from the conditions \eqref{2.2}.  We then
obtain the vertex condition
\begin{equation}
  \label{eq:vc_limiting}
  -u'(1) = \gamma u(1), \qquad \gamma \in \R \cup \{\infty\},
\end{equation}
where $\gamma = \infty$ should be interpreted as the Dirichlet
condition $u(1)=0$.  The dependence of $\gamma$ on the original
conditions at the vertex $\vc$ will not be important to our discussion.

Convergence to this limiting graph operator was established in
\cite{BerLatSuk_am19} under a sufficient ``non-resonance'' condition
\cite[Cond.~3.2]{BerLatSuk_am19}.  We remark that in the special case
of a graph with \emph{all} edges of order $\e$, the condition of
\cite[Cond.~3.2]{BerLatSuk_am19} was shown to be not only sufficient
but also necessary \cite{BerCdV_prep23}.  The
non-resonance\footnote{The name ``non-resonance'' was chosen due to an
  analogy to Sommerfeld radiation condition for resonances, as it seeks
  to exclude the situation where non-zero values on the short edges
  occur in the absence of any input from the order 1 edges.} condition
of \cite{BerLatSuk_am19} is formulated exclusively in terms of the
boundary values of the functions on edges.  In the present setting,
the condition can be formulated as follows: if
$u(1) = u'(1) = u'_\e(\e) = 0$ then the vertex conditions at $\vc$
should enforce that $u_\e(\e)=0$.

Direct inspection of all possible conditions at $\vc$ shows that the
only cases where \cite[Cond.~3.2]{BerLatSuk_am19} is \emph{not}
satisfied are
\begin{enumerate}
\item \label{item:1} $\rank\rP=1$, $z=\infty$, $\mu=0$,
\item \label{item:2} $\rank\rP=0$, $a=c=0$.
\end{enumerate}
In all other cases, \cite[Thm.~3.5]{BerLatSuk_am19} guarantees that
(using our present notation)
\begin{align}
  \label{eq:BLSi}
  & \cP_e (\Op_\e-\l)^{-1} \cP_e^* \ \to \
    (\Oplim-\l)^{-1},
    \qquad\text{in operator norm on }L^2(e),\\
  \label{eq:BLSii}
  & \cP_s (\Op_\e-\l)^{-1} \ \to\  0
    \qquad\text{in operator norm }L_2(\G_\e)\to L_2({s}_\e).
\end{align}
Equation~\eqref{eq:Re} in Theorem~\ref{th:res} of the present work
shows that \eqref{eq:BLSi} holds even when the non-resonance condition
above is violated.  In contrast, as can be seen from equation
\eqref{eq:Rell}, convergence in \eqref{eq:BLSii} holds \emph{if and
  only if} the non-resonance condition is satisfied.  Namely, cases
\eqref{item:1}-\eqref{item:2} above correspond to the third case in
\eqref{eq:Rell} with a leading term of order 1.  We remark that
because of the norm-distorting rescaling $\cS_\e$ in the definition of
$\cR_s$, the last case of \eqref{eq:Rell} actually corresponds to
$\| \cP_s (\Op_\e-\l)^{-1} \|_{L_2(\G_\e)\to L_2({s}_\e)} =
O(\sqrt{\e})$.

Furthermore, \cite[Thm.~3.6]{BerLatSuk_am19} establishes convergence
of spectra on every compact (again, under the non-resonance
condition).  In cases \eqref{item:1} and \eqref{item:2} above, the
graph decouples into the edge of length $\e$ with Neumann conditions
at both ends, $u'_\e(0) = u'_\e(\e)=0$, and the edge of length $1$.
The obstacle to the convergence of spectra is the constant
eigenfunction localized on the vanishing edge $s_\e$.  Notably, localization
of the eigenfunctions of type (S) and (C) on $s_\e$, see
Theorem~\ref{th:efs}, does not prevent convergence of spectra (on any
compact) since the corresponding eigenvalues escape to $-\infty$.

The behavior of the resolvents of general elliptic operators on
general graphs with small edges was studied in \cite{Bor_am22} under
a different non-resonant condition.
For our model this
condition is formulated as follows. Consider the operator
$\Op_\infty=-\frac{d^2\ }{dx^2}$ on the graph consisting of a finite edge ${s}$ and
a lead $e_\infty:=(0,+\infty)$, connected by the vertex $\vc$.  The other
end-point of the edge ${s}$ is denoted by $v_1$.  At $v_1$ the Neumann
condition is imposed, while the vertex condition at $\vc$ is obtained
by a suitable rescaling and taking the limit $\e\to0$, namely
\begin{equation*}
\rP \rU_\infty=0,\qquad \rQ \rU_\infty'=0, \qquad
\rU_\infty:= \begin{pmatrix}
u_{s}(1)
\\
u_\infty(0)
\boldsymbol{}\end{pmatrix},\qquad \rU_\infty':= \begin{pmatrix}
-u_{s}'(1)
\\
u_\infty'(0)
\end{pmatrix}.
\end{equation*}
Here $u_{s}$ is the restriction of a given function to ${s}$, while
$u_\infty$ is the restriction to $e_\infty$. Due to the presence of the lead
$e_\infty$, the operator $\Op_\infty$ has essential spectrum at
$[0,+\infty)$. The non-resonance condition from \cite{Bor_am22}
prohibits existence of an embedded eigenvalue at the bottom of this
essential spectrum.  In view of the simple structure of the operator
$\Op_\infty$, an embedded eigenvalue must correspond to an
eigenfunction which is constant on ${s}$ and identically zero on
$e_\infty$.  This is possible in the following two cases:
\begin{itemize}
\item $\rank\rP=1$, $z=\infty$,
\item $\rank\rP=0$.
\end{itemize}
Correspondingly, the non-resonance condition of \cite{Bor_am22} holds
in all cases except the above.

Under the non-resonance condition, Theorem~2.1 in \cite{Bor_am22}
guarantees that the resolvent is holomorphic and that the leading
terms of both $\cR_{s}$ and $\cR_e$ are governed only by $f_e$, see
\cite[Eq.~(2.30)]{Bor_am22}. As Theorem~\ref{th:res} of the present
work shows, the operators $\cR_{s}$ and $\cR_e$ are holomorphic in all
cases.  However, when the non-resonance condition is broken, the
leading term in the Taylor series for $\cR_{s}$ may involve a
functional depending on $f_{s}$, as seen in equation~\eqref{eq:Rell}.

As seen from Theorem~\ref{tm:asym} and Table~\ref{tab:neg_eig} above,
negative eigenvalues that are unbounded as functions of $\e$ can occur
when the non-resonance condition of \cite{Bor_am22} is violated; these
are eigenvalues of the type (S) or (C).  We also observe that,
according to (\ref{2.24}), (\ref{2.25}), the associated eigenfunctions
are either localized only on the short edge or both on the finite and
short edges. This is in contrast with the eigenfunctions associated
with the eigenvalues of type (B), which localize exclusively on the
edge of constant length, see (\ref{2.23}).

\section{Resolvent}

In this section we prove Theorem~\ref{th:res}.
Let $f=(f_{s},f_e)\in L_2({s})\oplus L_2(e)$ be an arbitrary function and denote
\begin{equation}\label{3.6}
 F:=(\cS_\e\oplus \cI)f=(\cS_\e f_{s},f_e), \qquad (u_{{s}_\e},u_e):=(\Op_\e-\l)^{-1}F.
\end{equation}
Comparing the above formulas with the  definition of the operators $\cR_{s}$ and $\cR_e$ in (\ref{2.3}), we see that
\begin{equation}\label{3.7}
\cR_{s}(\e,\l)f=\cS_\e^{-1}u_{{s}_\e},\qquad \cR_e(\e,\l)f=u_e.
\end{equation}

In view of the definition of the operator $\Op_\e$, the components of the function $(u_{{s}_\e},u_e)$ solve the problems
\begin{equation}\label{3.1}
\begin{aligned}
&
-\frac{d^2 u_{{s}_\e}}{dx_\e^2}-\l u_{{s}_\e} =\cS_\e f_{s} && \text{on}\quad {s}_\e,\quad && u_{{s}_\e}'(0)=0,
\\
&
-\frac{d^2 u_e}{dx^2}-\l u_e =f_e && \text{on}\quad e, && u_e(0)=0,
\end{aligned}
\end{equation}
and this is why they can be found explicitly:
\begin{equation}\label{3.2}
\begin{aligned}
u_e(x) & =-\frac{1}{\sqrt{\l}}\int\limits_{0}^{x} \sin\sqrt{\l}(x-t) f_e(t)\,dt + \cC_e \sin\sqrt{\l} x, && x\in e
\\
&=(\cL f_e)(x) + \cC_e \sin\sqrt{\l} x, \\
u_{{s}_\e}(x_\e)&=-\frac{1}{\sqrt{\l}}\int\limits_{0}^{x_\e} \sin\sqrt{\l}(x_\e-r) f_{s}\left(\frac{r}{\e}\right)\,dr + \cC_{s} \cos\sqrt{\l} x_\e
\\
&= \big(\cL^\e f_s\big)(y) + \cC_{s} \cos\e\sqrt{\l}y, && y:=\frac{x_\e}{\e}\in s,
\end{aligned}
\end{equation}
where the branch of the square root can be chosen arbitrarily,
$\cC_{s}$ and $\cC_e$ are some constants, $\cL$ is as defined
in~\eqref{eq:Ldef} and we also introduced
\begin{equation}
  \label{eq:Lepsilon}
  \big(\cL^\e f_s\big)(y) =
  -\frac{\e}{\sqrt{\l}}\int\limits_{0}^{y} \sin\e\sqrt{\l} (y-t)
  f_{s}(t)\,dt
  \ =\  O(\e^2).
\end{equation}

The vectors $\rU$ and $\rU'$ appearing in \eqref{2.2} can be
represented as $\rU=\rV -  \mathrm{W}$, $\rU'=\rV' - \mathrm{W}'$,
with
\begin{equation}\label{3.4}
\begin{aligned}
&\rV:=
\begin{pmatrix}
\cC_{s} \cos\sqrt{\l}\e
\\
\cC_e \sin\sqrt{\l}
\end{pmatrix},
&&\rV':=\sqrt{\l}\begin{pmatrix}
\cC_{s}  \sin\sqrt{\l}\e
\\
-\cC_e  \cos\sqrt{\l}
\end{pmatrix},
\\
&\mathrm{W}:=\frac{1}{\sqrt{\l}}
\begin{pmatrix}
\displaystyle\e\int\limits_{0}^{1} \sin\e\sqrt{\l}(1-t) f_{s}(t)\,dt
\\
\displaystyle\int\limits_{0}^{1} \sin\sqrt{\l}(1-t) f_e(t)\,dt
\end{pmatrix},
&&\mathrm{W}':=-
\begin{pmatrix}
\displaystyle\e\int\limits_{0}^{1} \cos\e\sqrt{\l}(1-t) f_{s}(t)\,dt
\\
\displaystyle\int\limits_{0}^{1} \cos\sqrt{\l}(1-t) f_e(t)\,dt
\end{pmatrix},
\end{aligned}
\end{equation}
and the boundary condition at $\vc$ in (\ref{2.2}) becomes
\begin{equation}\label{3.3}
\rP \rV= \rP\mathrm{W},\qquad
\rQ\rV'-\mathrm{TQ}\rV= \rQ \mathrm{W}'- \mathrm{TQ} \mathrm{W}.
\end{equation}
Once we solve this linear system of equations with respect to
$\cC_{s}$ and $\cC_e$, we will find the functions $u_{s}$ and $u_e$
and thus the resolvent through \eqref{3.7}.  At this point we observe
that $\mathrm{W}$ and $\mathrm{W}'$ are holomorphic in $\e$ as
operators from $L_2({s})\oplus L_2(e)$ to $\C^2$.  The coefficients
$\cC_{s}$ and $\cC_e$ will be shown to be linear combinations of the
entries of $\mathrm{W}$ and $\mathrm{W}'$, and therefore also holomorphic.

The solution of \eqref{3.3} depends on the rank of the projects $\rP$
and $\rQ$ and will be addressed case by case.

\subsection{Case $\rank\rP=2$}

This is the simplest case due to formulas (\ref{2.4}); we immediately
obtain $\rV=\mathrm{W}$ and
\begin{equation}\label{3.5}
  \cC_{s} =\frac{W_1}{\cos\sqrt{\l}\e},
  \qquad
  \cC_e = \frac{W_2}{\sin\sqrt{\l}},
\end{equation}
where $W_1$ and $W_2$ are the corresponding entries of the vector
$\mathrm{W}$, equation~\eqref{3.4}.  These identities and (\ref{3.2}),
(\ref{3.7}) yield:
\begin{equation}\label{3.8}
\begin{aligned}
(\cR_{s}(\e,\l)f)(y)=&\big(\cL^\e f_s\big)(y)
+\frac{\e\cos\sqrt{\l} \e y}{\sqrt{\l}\cos\sqrt{\l}\e}\int\limits_{0}^{1} \sin\e\sqrt{\l}(1-t) f_{s}(t)\,dt,\quad && y\in{s},
\\
(\cR_e(\e,\l)f)(x)=&\,(\cL f_e)(x) + \frac{\sin\sqrt{\l} x}{\sqrt{\l}\sin\sqrt{\l}} \int\limits_{0}^{1} \sin\sqrt{\l}(1-t) f_e(t)\,dt, \qquad &&x\in e.
\end{aligned}
\end{equation}
The obtained formulas imply that the operators $\cR_{s}$ and $\cR_e$
are holomorphic in $\e$, the latter is even independent of $\e$, and
relations (\ref{eq:Re}), (\ref{eq:Rell}) are satisfied with
\begin{equation}
  \label{eq:Bf_e}
  \cB f_e = \frac{1}{\sqrt{\l}\sin\sqrt{\l}}
  \int\limits_{0}^{1} \sin\sqrt{\l}(1-t) f_e(t)\,dt.
\end{equation}

\subsection{Case $\rank \rP=1$}

Substituting representations (\ref{eq:parametrization}) into
(\ref{3.3}), we solve this linear system of equations:
\begin{equation}\label{3.10}
\begin{aligned}
  &\cC_{s}
  = \frac{(W_1+zW_2) \sqrt{\l}\cos\sqrt{\l}
    + \big(\mu(1+|z|^2)W_1 - |z|^2 W_1' + zW_2'\big) \sin\sqrt{\l}
    }{g_0(\e,\sqrt{\l})},\\
    &\cC_e
    = \frac{\big(\mu(1+|z|^2)W_2 + \overline{z} W_1' -
      W_2'\big) \cos\sqrt{\l}\e
      - \overline{z}(W_1+zW_2) \sqrt{\l}\sin\sqrt{\l}\e}
    {g_0(\e,\sqrt{\l})},
\end{aligned}
\end{equation}
where
\begin{equation}
  \label{eq:g0def}
  g_0(\e,k):=(k\cos k+\mu\sin k)\cos k \e - |z|^2 (k\sin k\e - \mu\cos k\e)\sin k.
\end{equation}
The function $g_0(\e,k)$ is clearly holomorphic in $\e$ and has zeros which correspond to the poles of the resolvent. These poles are the square roots of eigenvalues of $\Op_\e$, which must be real by self-adjointness. With $k=\sqrt{\l}\notin\mathbb{R}$, we see that $g_0(\e,\sqrt{\l})\neq0$, and for sufficiently small $\e$, we have
\begin{equation*}
g_0(\e,k)=k\cos k + (1+|z|^2)\mu\sin k+O(\e).
\end{equation*}
Hence, the functions $\cC_{s}$ and $\cC_e$, regarded as functionals on $L_2({s})\oplus L_2(e)$,  are holomorphic in $\e$. Returning back to functions $u_{s}$ and $u_e$ in (\ref{3.2}) and using formulas (\ref{3.7}), we conclude that the operators $\cR_{s}$ and $\cR_e$ are holomorphic in $\e$ in this case, too.
By straightforward calculations we see that the leading terms of their Taylor series are given by  (\ref{eq:Re}), (\ref{eq:Rell}).

In the subcase $z=\infty$, formulas (\ref{3.10}) become
\begin{equation}\label{3.20}
\begin{aligned}
&\cC_{s}= \frac{W'_1-\mu W_1}{\sqrt{\l}\sin\sqrt{\l}\e - \mu\cos \sqrt{\l}\e}, \qquad
&\cC_e=\frac{W_2}{\sin\sqrt{\l}},
\end{aligned}
\end{equation}
and together with (\ref{3.2}), (\ref{3.7}) they prove  (\ref{eq:Re}), (\ref{eq:Rell}).

\subsection{Case $\rank \rP=0$}

Due to (\ref{2.5}) the first equation in system (\ref{3.3}) becomes
trivial and solving the other we find:
 \begin{equation}\label{3.12}
 \begin{aligned}
   &\cC_{s}
   = \frac{ \big(-W_1'+aW_1+cW_2) \sqrt{\l}\cos\sqrt{\l}
     + \big(-bW_1'+cW_2'+(ab-|c|^2)W_1\big) \sin\sqrt{\l}}
    {h_0(\e,\sqrt{\l})},
     \\
     &\cC_e
     = \frac{\big(\overline{c}W_1'-aW_2'+(ab-|c|^2)W_2\big)
       \cos\sqrt{\l}\e
      + \big(W'_2 - \overline{c}W_1 - bW_2\big) \sqrt{\l}\sin
      \sqrt{\l}\e}
    {h_0(\e,\sqrt{\l})},
 \end{aligned}
\end{equation}
where
\begin{equation}\label{3.13}
  h_0(\e,k) := -(k\sin k\e-a\cos k\e)(k\cos k+b\sin k)
  - |c|^2\sin k\cos k\e.
\end{equation}
The function $h_0(\e,k)$ is obviously holomorphic in $\e$ and as we noticed earlier with $g_0(\e,k)$, its roots must correspond to real values of $k$. Since $k=\sqrt{\l}\notin\mathbb{R}$, we find that $h_0(\e,\sqrt{\l})\neq0$ for $\varepsilon>0$. As $\e\to+0$, the leading terms of its Taylor expansion are
\begin{equation}\label{3.14}
h_0(\e,k)=(ab-|c|^2)\sin k+a k\cos k -\e  k^2  (k\cos k+b\sin k) +O(\e^2).
\end{equation}
If
\begin{equation}\label{3.15}
 ab-|c|^2\ne0\qquad\text{or}\qquad a\ne0,
\end{equation}
then the leading term in the above expansion is non-zero.  The
coefficients $\cC_{s}$ and $\cC_e$ and, consequently, the operators
$\cR_{s}$ and $\cR_e$ again holomorphic in $\e$.  The leading terms of
the Taylor expansion of $\cR_{s}$ and $\cR_e$ can be found by
straightforward calculations, leading to (\ref{eq:Re}),
(\ref{eq:Rell}).

If condition \eqref{3.15} does not hold, i.e. if
\begin{equation}\label{3.17}
  a=0\qquad\text{and}\qquad c=0,
\end{equation}
the leading term in expansion \eqref{3.14} disappears, but the order $\e$
term is still non-zero.  In this case formulas
(\ref{3.12}) simplify:
\begin{align*}
\cC_{s}=\frac{W_1'}{\sqrt{\l}\sin \sqrt{\l}\e}, \qquad
\cC_e=\frac{-W_2'+bW_2}{\sqrt{\l}\cos\sqrt{\l} +b\sin\sqrt{\l}}.
\end{align*}
The leading terms of the Taylor series in $\e$ of the the holomorphic
operators $\cR_e$ and $\cR_{s}$ are found by straightforward
calculations to be given by (\ref{eq:Re}), (\ref{eq:Rell}).

\section{Eigenvalues and eigenfunctions}

In this section we establish Theorems~\ref{tm:asym} and \ref{th:efs}. The resolvent of the operator $\Op_\e$ has poles at its eigenvalues. In view of formula (\ref{2.9}) for this resolvent, we conclude that these poles can appear only as the poles of the operators $\cR_{s}$ and $\cR_e$.  The formulas for the operators $\cR_{s}$ and $\cR_e$ obtained in the previous section show that such poles should coincide with the roots of the equations
\begin{equation}\label{4.1}
\begin{aligned}
&\cos k\e \sin k=0 \quad&&\text{if}\quad \rank\rP=2,
\\
&g_0(\e,k)=0 &&\text{if}\quad \rank\rP=1,
\\
&h_0(\e,k)=0 &&\text{if}\quad \rank\rP=0,
\end{aligned}
\end{equation}
where we have denoted $k:=\sqrt{\l}$.

Since $\l=k^2$ and we are interested in negative eigenvalues, we seek $k=\iu\kappa$, where $\kappa\in\mathbb{R}$ and $\kappa>0$. Then equations (\ref{4.1}) become 
\begin{align}\label{4.2}
&\cosh \k\e\sinh\k=0 &&\text{if}\quad \rank\rP=2,
\\
&(\k\cosh\k+\mu\sinh\k)\cosh \k \e + |z|^2 (\k\sinh \k\e + \mu\cosh \k\e)\sinh \k=0 &&\text{if}\quad \rank\rP=1,\label{4.3}
\\
&(\k\sinh\k\e + a\cosh\k\e)(\k\cosh\k + b\sinh\k)
- |c|^2\sinh\k\cosh\k\e=0&&\text{if}\quad \rank\rP=0.
\label{4.4}
\end{align}
We also mention that these equations can be obtained by a straightforward analysis of the eigenvalue equation for the operator $\Op_\e$.

It is easy to see that equation (\ref{4.2}) has no positive roots for each $\e$ and hence, the operator $\Op_\e$ possesses no negative eigenvalues in the case $\rank\rP=2$.

\subsection{Case $\rank\rP=1$}\label{sec4.1}

We divide equation (\ref{4.3}) by $\cosh \k\e \sinh \k$ and we get an equivalent equation
\begin{equation}\label{4.5}
\k\coth \k + |z|^2 \k\tanh \k\e=-\mu (1+|z|^2).
\end{equation}
The function in the left hand side of this equation is monotone in $\k>0$ and hence, its minimum is attained at $\k=0$ and it is equal to $1$. Therefore, equation (\ref{4.5}) has no positive roots as $-\mu(1+|z|^2)\leqslant 1$ and it possesses a unique positive root $\k=\k(\e)$ as $-\mu(1+|z|^2)>1$.
Hence, for finite $z$ the operator $\Op_\e$ possesses negative eigenvalues only if $-\mu(1+|z|^2)>1$ and in this case it has just a single eigenvalue.

By the implicit function theorem for holomorphic functions
\cite[Thm.~1.3.5 and Rem.~1.3.6]{Narasimhan_analysis_manifolds} we immediately conclude that the root $\k(\e)$ is holomorphic in $\e$ and $\k(0)=\k_1$, where $\k_1$ is the unique root of the equation
\begin{equation}\label{3.19}
\k \coth \k=-\mu(1+|z|^2).
\end{equation}



For this eigenvalue $\l(\e)=-\k^2(\e)$ of type (B), we seek to determine the norm of the corresponding eigenfunction $\psi = (\psi_s,\psi_e)\in L^2({s}_\e)\oplus L^2(e)$, which can be represented as
\begin{equation}\label{4.37}
\psi_s = \cC_s \cosh\k x_\e, \qquad \psi_e = \cC_e \sinh\k x
\end{equation}
where the coefficients $\cC_{s}$ and $\cC_e$ are determined by the boundary conditions at the central vertex $v_c$. There are two independent conditions to be satisfied, and for the case where $\rank\rP=1$, one condition comes from each of the last two equations of (\ref{2.2}). However, we need only impose
\begin{equation}\label{4.38}
\psi_s(\e) + z\psi_e(1)=0,
\end{equation}
because the other condition is then automatically satisfied by nature of $\psi$ being an eigenfunction. This leads to
\begin{equation}\label{4.14}
\cC_{s}=-\b_1(\k,\e)z\sinh\k,\qquad \cC_e=\b_1(\k,\e)\cosh\e\k,
\end{equation}
where $\beta_1$ is determined by the normalization
\begin{equation}\label{4.41}
\|\psi\|^2=\|\psi_s\|^2+\|\psi_e\|^2=1. 
\end{equation}
By straightforward calculation, we see that
\begin{equation}\label{4.33}
\|\psi_s\|^2=\frac{\sinh2\k\e+2\k\e}{4\k} |\cC_{s}|^2, \qquad
\|\psi_e\|^2=\frac{\sinh2\k-2\k}{4\k} |\cC_e|^2.
\end{equation}
Combining this with (\ref{4.14}) and choosing appropriate $\b_1$, we find that
\begin{equation}
\|\psi_s\|^2=\frac{(\sinh2\k\e+2\k\e)|z|^2\sinh^2\k}{(\sinh2\k\e+2\k\e)|z|^2\sinh^2\k+(\sinh2\k-2\k)\cosh^2\k\e}.
\end{equation}
We then apply $\k(\e)=\k_1+O(\e)$ to obtain (\ref{2.23}), where $\|\psi_e\|^2$ is most easily found from (\ref{4.41}).

If $z=\infty$, then equation (\ref{4.5}) is to be rewritten as
\begin{equation}\label{4.6}
\k\tanh \k\e=-\mu.
\end{equation}
For non-negative $\mu$ it has no positive solution and in this case the operator $\Op_\e$ possesses no negative eigenvalues. For negative $\mu$ we make the change $\tau:=\k\e$ and rewrite equation (\ref{4.6}) as
\begin{equation}\label{4.8}
\tau \tanh\tau=-\e\mu.
\end{equation}
In view of the Taylor series for the function $\tau\mapsto \tau\tanh\tau$ about zero, this function can be represented as $\tau\tanh\tau=T(\tau^2)$, where $T=T(t)$ is
a holomorphic function in some fixed neighborhood of the origin in the complex  plane and $T(0)=0$. Then we can rewrite equation (\ref{4.8}) as
\begin{equation*}
T(t)=-\e\mu,\qquad t:=\tau^2,
\end{equation*}
and by the implicit function theorem \cite[Thm.~1.3.5 and
Rem.~1.3.6]{Narasimhan_analysis_manifolds} we immediately see that
this equation possesses a unique root $t=\e\mu t_0(\e\mu)$, where
$t_0$ is holomorphic at zero and
$t_0(0)=0$. 
Returning back to equation (\ref{4.6}), we see that it also possesses a unique root $\k(\e)$ such that the function $\k^2(\e)$ is meromorphic in $\e\mu$.
Hence, in the considered case the operator $\Op_\e$ possesses a unique negative eigenvalue $\l(\e)=-\k^2(\e)$ of type (S), which is meromorphic in $\e$. In this case, the associated eigenfunction is determined by $\cC_e=0$, and (\ref{2.24}) holds because $\|\psi_s\|^2=1$ independent of $\e$.

\subsection{Case $\rank\rP=0$ preliminaries}

In the considered case we again suppose that $k=\iu\k$, then divide equation (\ref{4.4}) by
$\cosh\k\e \sinh k$ and this results in the equation
\begin{equation}\label{4.7}
(\k\tanh \k\e+a)(\k\coth\k+b) - |c|^2=0,
\end{equation}
and the associated eigenfunctions satisfy (\ref{4.37}) and (\ref{4.33}).

The study of equation (\ref{4.7}) is more complicated than that of
(\ref{4.5}) and here it is convenient to know a priori the number of
its positive roots depending on $a$, $b$ and $c$, that is, the number
of the negative eigenvalues of the operator $\Op_\e$. The latter can
be found by using the Behrndt--Luger formula \cite{BehLug_jpa10}.

\begin{lem}\label{lem:number_eig_rank2}
  The operator $\Op_\e$ with vertex conditions given by
  \eqref{2.2}-\eqref{2.3} and \eqref{2.5} has
  \begin{enumerate}
  \item\label{2eig} two negative eigenvalues if
    \begin{equation}\label{2.14}
      |c|^2 - ab < a < 0;
    \end{equation}
    
  \item\label{1eig} one negative eigenvalue if
    \begin{equation}\label{2.15}
      |c|^2 -ab > a
    \end{equation}
    or
    \begin{equation}\label{2.16}
      |c|^2 = a(b+1),\qquad a+b+1<0;
    \end{equation}
    
  \item no negative eigenvalues otherwise.
  \end{enumerate}
\end{lem}

\begin{proof}
  We introduce the matrices
  \begin{equation*}
    \rA =  \begin{pmatrix}
      0 & 0 & 0 & 0 \\
      0 & -a & -c & 0 \\
      0 & -\bar{c} & -b & 0
      \\
      0 & 0 & 0 & 1
    \end{pmatrix}, \qquad
    \rB =\begin{pmatrix}
      1 & 0 & 0 & 0 \\
      0 & 1 & 0 & 0 \\
      0 & 0 & 1 & 0 \\
      0 & 0 & 0 & 0
    \end{pmatrix}, \qquad \rM:=
    \begin{pmatrix}
      -\e^{-1} & \e^{-1} & 0 & 0
      \\
      \e^{-1} & -\e^{-1} & 0 & 0
      \\
      0 & 0 & -1 & 1
      \\
      0 & 0 & 1 & -1
    \end{pmatrix}.
  \end{equation*}
  Here $A$ and $B$ encode the vertex conditions in our graph (see
  \cite{KosSch_jpa99}), while $\rM$ represents the
  Dirichlet-to-Neumann operator at $\lambda=0$ (see
  \cite[Sec~3.5]{BerKuc_graphs}).  According to
  \cite[Thm.~1]{BehLug_jpa10}, the number of the negative eigenvalues
  of the operator $\Op_\e$ is given by the number of the positive
  eigenvalues of the matrix
  \begin{equation*}
    \rD:= \rA\rB^*+\rB\rM\rB^* = \begin{pmatrix}
      -\e^{-1} & \e^{-1} & 0 & 0 \\
      \e^{-1} & -a-\e^{-1} & -c & 0 \\
      0 & -\bar{c} & -b-1 & 0 \\
      0 & 0 & 0 & 0
    \end{pmatrix}.
  \end{equation*}
  It is obvious that the number of positive eigenvalues of $D$
  coincides with that of
  \begin{equation*}
    \rD_\e:=\rD_\infty-\e^{-1} \rE_0,\qquad \rD_\infty:= \begin{pmatrix}
      0 & 0 & 0  \\
      0 & -a & -c  \\
      0 & -\bar{c} & -b-1
    \end{pmatrix}, \qquad \rE_0:=\begin{pmatrix}
      1 & -1 & 0   \\
      -1 & 1 & 0  \\
      0 & 0 & 0
    \end{pmatrix}.
  \end{equation*}
  
  The main point of the proof is that the matrix $\rD_\e$ has the same
  number of positive eigenvalues as $\rD_\infty$, which will follow
  from the Eigenvalue Interlacing Theorem for rank-one perturbations
  \cite[Cor.\ 4.3.9]{HornJohnson}, namely that
  \begin{equation}
    \label{eq:interlacing}
    \lambda_i(\rD_\infty) \leqslant \lambda_{i+1}(\rD_\e) \leqslant \lambda_{i+1}(\rD_\infty),
    \qquad
    i = 1,2,
  \end{equation}
  where the eigenvalues are numbered in the ascending order counting the
  multiplicities. From this inequality we conclude immediately that
  $\rD_\e$ has \emph{at most} the number of positive eigenvalues of
  $\rD_\infty$.

  Furthermore, the cases of equality in \eqref{eq:interlacing} are
  characterized conveniently as follows
  \cite[Thm.~4.3]{BerKenKurMug_tams19}: if a given value $\lambda$ has
  multiplicities $m_0$ and $m_1$ in the spectra of $\rD_\infty$ and
  $\rD_\e$, then $|m_0-m_1|\leqslant 1$ and the intersection of the
  corresponding eigenspaces has dimension $\min(m_0,m_1)$.

  Consider first the case when $\rD_\infty$ has two positive
  eigenvalues.  This occurs when the non-trivial $2\times2$ submatrix
  of $\rD_\infty$ has both its determinant and trace positive, i.e.\
  $a(b+1)-|c|^2>0$ and $a+b+1 < 0$.  Inequality \eqref{eq:interlacing}
  gives $\lambda_2(\rD_\e) \geq \lambda_1(\rD_\infty)=0$.  Moreover,
  $\lambda_2(\rD_\e)=0$ would mean that
  $\Ker \rD_\infty \subset \Ker \rD_\e$; since the former is the span
  of $(1,0,0)^T$, we can exclude this possibility, obtaining
  $\lambda_3(\rD_\e) \geq \lambda_2(\rD_\e)>0$.  We remark that
  $a(b+1)-|c|^2>0$ implies $a(b+1)>0$ and therefore $a+b+1 < 0$ is
  equivalent to $a<0$.

  Suppose now that $\rD_\infty$ has one positive and one negative
  eigenvalue, which occurs when $a(b+1)-|c|^2 < 0$.  Inequality
  \eqref{eq:interlacing} gives
  $\lambda_3(\rD_\e) \geq \lambda_2(\rD_\infty)=0$ and we can exclude
  the case of equality similarly to above.

  Finally, if $\rD_\infty$ has one positive and two zero eigenvalues,
  i.e.\ if $|c|^2 = a(b+1)$ and $a+b+1<0$, inequality
  \eqref{eq:interlacing} gives
  $\lambda_3(\rD_\e) \geq \lambda_2(\rD_\infty)=\lambda_2(\rD_\e) =
  \lambda_1(\rD_\infty)=0$.  If $\lambda_3(\rD_\e)$ were equal to 0,
  the multiplicity of zero would be at least 2 in the spectrum of
  $\rD_\e$, and we again have $\Ker \rD_\infty \subset \Ker \rD_\e$,
  which is impossible.  The proof is complete.
\end{proof}

It follows immediately from this lemma that the operator $\Op_\e$ can
have negative eigenvalues only under the conditions formulated in
items~\eqref{2eig} and~\eqref{1eig} of
Lemma~\ref{lem:number_eig_rank2}. This means that under these
conditions equation (\ref{4.7}) can have respectively either one or
two positive roots.

\subsection{Case $\rank\rP=0$, two negative eigenvalues of $\Op_\e$}

We first suppose that inequalities (\ref{2.14}) are satisfied
therefore equation (\ref{4.7}) possesses two roots. Setting $\e=0$, this
equation becomes
\begin{equation}\label{4.9}
\Ups_0(\k)=\frac{|c|^2-a(b+1)}{a},\qquad \text{where }\Ups_0(\k):=\k\coth \k-1.
\end{equation}
The function $\Ups_0(\k)$ is monotonically increasing in $\k\in\mathbb{R}$ and vanishes at $\k=0$, while the right hand side in the above equation is positive by our assumptions. Hence, this equation possesses a unique root $\k_0>0$. We also observe that the function $\Ups_0(\k)$ is holomorphic in some fixed ball in the complex plane centered at the point $\k_0$.

We rewrite equation (\ref{4.7})  as
\begin{equation}\label{4.10}
\Ups_0(\k)+ \frac{1}{a}(\Ups_0(\k)+b+1)\k\tanh \k\e=\frac{|c|^2-a(b+1)}{a}
\end{equation}
and the left hand side of this equation is holomorphic in $\k$ in the aforementioned ball centered at $\k_0$ and sufficiently small $\e$. Hence, by the implicit function theorem \cite[Thm.~1.3.5 and Rem.~1.3.6]{Narasimhan_analysis_manifolds} we immediately conclude that this equation possesses a unique root $\k=\k(\e)$, which is holomorphic in $\e$. This root then generates a negative eigenvalue $\l(\e)=-\k^2(\e)$ of type (B).

%

Next, we seek the second root of equation (\ref{4.7}) as $\k=\e^{-\frac{1}{2}}\rho$ and, multiplying (\ref{4.7}) by $\e^{\frac{1}{2}}$, for the new unknown $\rho$ we get the equation
\begin{equation}\label{4.17}
\big(\e^{-\frac{1}{2}}\rho\tanh \rho\e^\frac{1}{2}+a\big) \big(\rho\coth\e^{-\frac{1}{2}}\rho+\e^\frac{1}{2}b\big) - \e^{\frac{1}{2}}|c|^2=0.
\end{equation}
Since we seek positive roots of equation (\ref{4.7}), we do the same for the above equation. It is clear that
\begin{equation}\label{4.18}
\coth t=1 + \Ups_1(t),\qquad \text{where }\Ups_1(t):=\frac{2 e^{-2t}}{1-e^{-2t}},
\end{equation}
and equation (\ref{4.17}) can be represented as
\begin{equation}\label{4.19}
\Ups_2(\rho,\e^\frac{1}{2})=0,
\end{equation}
where
\begin{align*}
&\Ups_2(\rho,\tht):=\Ups_3(\rho,\tht) +\big(\tht^{-1}\rho\tanh \rho\tht^{-1}+a\big) \rho\Ups_1(\tht^{-1}\rho),
\\
&\Ups_3(\rho,\tht):=\big(\tht^{-1}\rho\tanh \rho\tht+a\big) \big(\rho+\tht b\big) - \tht |c|^2.
\end{align*}
The function $(\rho,\tht)\mapsto \tht^{-1}\rho\tanh \rho\tht$ is obviously holomorphic in $(\rho,\tht)$ as a function of two complex variables on the domain $\Om:=\{(\rho,\tht):\ \RE\rho>-\d,\ |\IM\rho|<\d,\ |\tht|<\d\}$ for some fixed small $\d$. For small $\tht$ the leading term  of its Taylor series is
\begin{equation*}
\tht^{-1}\rho\tanh \rho\tht=\rho^2+O(\tht^2).
\end{equation*}
Hence, the function $\Ups_3(\rho,\tht)$ is also holomorphic in $(\rho,\tht)\in\Om$. Since $\Ups_3(\rho,0)=\rho (\rho^2+a)$, the function $\Ups_3(\rho,0)$ possesses the only positive root $(-a)^\frac{1}{2}$ and by implicit function theorem \cite[Thm.~1.3.5 and Rem.~1.3.6]{Narasimhan_analysis_manifolds} we conclude that the equation $\Ups_3(\rho,\tht)=0$ possesses the only positive root $\rho_0=\rho_0(\tht)$, which is holomorphic in $\tht$ and
\begin{equation}\label{4.20}
\rho_0(\tht)=(-a)^\frac{1}{2} +O(\tht).
\end{equation}
Calculating the derivative
$\frac{\p \Ups_3}{\p \rho}(\rho,\tht)$, we see that
\begin{equation}\label{4.21}
\frac{\p \Ups_3}{\p \rho}(\rho,\tht)
\geqslant \frac{-a}{4}>0 \qquad\text{for}\qquad \frac{2(-a)^\frac{1}{2}}{3}\leqslant \rho\leqslant  \frac{4(-a)^\frac{1}{2}}{3}
\end{equation}
provided $\tht$ is real and small enough. It also follows from the definition of the function $\Ups_1(t)$ that
\begin{equation}\label{4.22}
0< \big|\tht^{-1}\rho\tanh \rho\tht+a\big|
\rho\Ups_1(\tht^{-1}\rho)< 3|a|^{\frac{3}{2}}e^{-2\tht^{-1}(-a)^\frac{1}{2}}
\quad\text{for}\quad \frac{2(-a)^\frac{1}{2}}{3}\leqslant \rho\leqslant  \frac{4(-a)^\frac{1}{2}}{3}
\end{equation}
if $\tht$ is real and small enough.
Using this estimate and (\ref{4.21}), by straightforward calculations for $\rho_\pm(\tht):=\rho_0(\tht)\pm 16 (-a)^\frac{1}{2}e^{-2\tht^{-1}(-a)^\frac{1}{2}}$ with small real $\tht$ we find:
\begin{equation}\label{4.29}
\Ups_2(\rho_+(\tht),\tht)  \geqslant 4 |a|^\frac{3}{2}e^{-2\tht^{-1}(-a)^\frac{1}{2}} - 3|a|^\frac{3}{2}e^{-2\tht^{-1}(-a)^\frac{1}{2}}>0,\qquad
\Ups_2(\rho_-(\tht),\tht)<0.
\end{equation}
Hence, equation (\ref{4.19}) possesses a root in the interval $\big(\rho_-(\tht),\rho_+(\tht)\big)$. Returning back to equation (\ref{4.7}), we then conclude that its second root reads as
\begin{equation}\label{4.23}
\k(\e)=\e^{-\frac{1}{2}} \big(\rho_0(\e^\frac{1}{2}) +e^{-2\e^{-\frac{1}{2}}(-a)^\frac{1}{2}}
\rho_1(\e)\big)
= (-a)^{\frac{1}{2}}\e^{-\frac{1}{2}} + O(1),
\end{equation}
where $\rho_1$ is some function with $|\rho_1(\e)|\leqslant 16$. This root produces a negative eigenvalue $\l(\e)=-\k^2(\e)$ of type (S).



\subsection{Case $\rank\rP=0$, one negative eigenvalue of $\Op_\e$}

Suppose now that either inequality (\ref{2.15}) or conditions
(\ref{2.16}) are satisfied and therefore the operator $\Op_\e$
possesses one negative eigenvalue.  Here we consider several cases.

\subsubsection{Assume that $a>0$}
Then we rewrite equation (\ref{4.7}) to (\ref{4.10}) and as above, we see that it possesses a root $\k(\e)$ holomorphic in $\e$, which produces an eigenvalue of type (B).

\subsubsection{Assume that $a<0$ and $|c|^2-a(b+1)\geqslant 0$} Then we seek the root of equation (\ref{4.7}) as $\k=\e^{-\frac{1}{2}}\rho$ and for $\rho$ we obtain equation (\ref{4.17}). As above, this equation possesses the root $\k(\e)$ obeying (\ref{4.23}) with a holomorphic function $\rho_0$ and identity (\ref{4.20}) is satisfied. This root produces an eigenvalue of type (S).

\subsubsection{Assume that $a=0$, $c\ne 0$ and, consequently, $b+1>0$}
We seek the root of  (\ref{4.7}) as $\k=\e^{-\frac{1}{3}}\tau$ and for $\tau$ we obtain the equation
\begin{equation}\label{4.25}
\e^{-\frac{2}{3}}\tau\tanh \tau\e^\frac{2}{3} (\tau\coth\e^{-\frac{1}{3}}\tau+\e^\frac{1}{3}b) - |c|^2=0.
\end{equation}
This equation can be analyzed in the same way as was done for equation (\ref{4.17}). Namely, we rewrite it as
\begin{equation}\label{4.26}
\Ups_4(\tau,\e^\frac{1}{3})=0,
\end{equation}
where
\begin{align*}
&\Ups_4(\tau,\nu):=
\Ups_5(\tau,\nu) + \nu^{-2}\tau^2\tanh \tau\nu^2 \Ups_1(\nu^{-1}\tau),
\\
&\Ups_5(\tau,\nu):=(\tau+\nu b) \nu^{-2}\tau\tanh \tau \nu^2 - |c|^2.\nonumber
\end{align*}
The function $\Ups_5$ is obviously holomorphic in $(\tau,\nu)\in\Om$ and $\Ups_5(\tau,0)=\tau^3-|c|^2$. The latter function possesses the only positive root $|c|^\frac{2}{3}$ and by the implicit function theorem we again conclude that the function $\Ups_5(\tau,\nu)$ has the only positive root $\tau_0(\nu)$, which is holomorphic in $\nu$ and
\begin{equation}\label{4.27}
\tau_0(\nu)=|c|^\frac{2}{3}-\frac{b\nu}{3} + O(\nu^2),\qquad \nu\to0.
\end{equation}
We also have estimates similar to (\ref{4.21}), (\ref{4.22}):
\begin{equation}\label{4.28}
\frac{\p \Ups_5}{\p \tau}(\tau,\nu)
\geqslant \frac{|c|^\frac{4}{3}}{2}>0,
\qquad |\nu^{-2}\tau\tanh \tau\nu^2 \Ups_1(\nu^{-1}\tau)|\leqslant 4|c|^2 e^{-2\nu^{-1}|c|^\frac{2}{3}}
\end{equation}
as $\nu$ is small enough and
\begin{equation*}
\frac{2^\frac{1}{2}}{3^\frac{1}{2}}|c|^\frac{2}{3}
\leqslant \tau\leqslant  \frac{2}{3^\frac{1}{2}}|c|^\frac{2}{3}.
\end{equation*}
As in (\ref{4.29}), we also confirm that
\begin{equation}\label{4.31}
\pm \Ups_4(\tau_\pm(\nu),\nu) >0,\qquad \tau_\pm(\nu):=\tau_0(\nu) \pm 9|c|^\frac{2}{3} e^{-2\nu^{-1}|c|^\frac{2}{3}}
\end{equation}
and hence, equation (\ref{4.26}) possesses a root on the interval $\big(\tau_-(\e^{\frac{1}{3}}), \tau_+(\e^{\frac{1}{3}})\big)$. Returning back to equation (\ref{4.7}), we conclude that its root can be represented as
\begin{equation}\label{4.32}
\k(\e)=\e^{-\frac{1}{3}} \big(\tau_0(\e^\frac{1}{3}) +e^{-2\e^{-\frac{1}{3}}|c|^\frac{2}{3}}
\tau_1(\e)\big)
= |c|^{\frac{2}{3}}\e^{-\frac{1}{3}}+O(1),
\end{equation}
where $\tau_1$ is some function with $|\tau_1(\e)|\leqslant 9$. This root produces an eigenvalue $\l(\e)=\k^2(\e)$ of type (C).

\subsubsection{Assume that $|c|^2-a(b+1)=0$ and $a=0$}
Then $c=0$ and $b+1<0$ and equation (\ref{4.7}) becomes (\ref{4.9}) and it possesses a unique fixed positive root $\k_0$, which is obviously of type (B).

\subsection{Case $\rank\rP=0$, eigenfunction localization}
On each edge, we now seek the norm of the eigenfunctions associated with the eigenvalues $-\k_B^2(\e)$ of type (B), $-\k_S^2(\e)$ of type (S), and $-\k_C^2(\e)$ of type (C). From the final equation of (\ref{2.2}), we see that the eigenfunction $\psi$ represented by (\ref{4.37}) must satisfy
\begin{equation}\label{4.39}
\psi_s'(\e)+a\psi_s(\e)+c\psi_e(1)=0,
\end{equation}
and the other vertex condition is again automatically satisfied because $\psi$ is an eigenfunction. This leads to
\begin{align}
\cC_{s}=-\b_0(\k,\e)c\sinh\k,\qquad \cC_e=\b_0(\k,\e)(\k\sinh\k\e+a\cosh\k\e),
\end{align}
where $\beta_0$ is determined by normalization. We substitute this into (\ref{4.33}), choose appropriate $\b_0$, and divide the numerator and denominator by $\sinh^2\k$ to find that
\begin{equation}\label{4.40}
\|\psi_s\|^2=\frac{(\sinh2\k\e+2\k\e)|c|^2}{(\sinh2\k\e+2\k\e)|c|^2+\Phi(\k)(\k\sinh\k\e+a\cosh\k\e)^2},
\end{equation}
where
\begin{equation*}
\Phi(\k)=2\coth\k-\frac{2\k}{\sinh^2\k},
\end{equation*}
and $\|\psi_e\|^2$ is found from (\ref{4.41}).

With $\k_B(\e)=\k_0+O(\e)$, we immediately obtain (\ref{2.23}) for eigenvalues of type (B). For $\k_S$ and $\k_C$, we have $\k$ tending to $\infty$ for small $\e$, so we first notice that in these cases $\Phi(\k)=2+O(e^{-\k})$. Then we apply (\ref{4.23}) to obtain (\ref{2.24}) for eigenvalues of type (S). Recall that eigenvalues of type (C) occur only for $a=0$, so in this case we apply (\ref{4.32}) to obtain (\ref{2.25}).

\section*{Acknowledgments}

The authors thank the anonymous referee for numerous improving suggestions.

The research by D.I.~Borisov was supported by Russian Science Foundation, grant no. 23-11-00009, {https://rscf.ru/project/23-11-00009/}.

The authors have no competing interests to declare that are relevant to the content of this article.

\bibliographystyle{abbrv}
\bibliography{bk_bibl}

\def\cprime{$'$} \def\cprime{$'$} \def\cprime{$'$} \def\cprime{$'$}
  \def\cprime{$'$} \def\cprime{$'$} \def\cprime{$'$}
  \def\polhk#1{\setbox0=\hbox{#1}{\ooalign{\hidewidth
  \lower1.5ex\hbox{`}\hidewidth\crcr\unhbox0}}} \def\cprime{$'$}
  \def\cprime{$'$}
\begin{thebibliography}{10}

\bibitem{BehLug_jpa10}
J.~Behrndt and A.~Luger.
\newblock On the number of negative eigenvalues of the {L}aplacian on a metric
  graph.
\newblock {\em J. Phys. A}, 43(47):474006, 11, 2010.

\bibitem{BerCdV_prep23}
G.~Berkolaiko and Y.~Colin~de Verdi\`ere.
\newblock Exotic eigenvalues of shrinking metric graphs.
\newblock preprint {\tt arXiv:2306.00631}, 2023.

\bibitem{BerKenKurMug_tams19}
G.~Berkolaiko, J.~B. Kennedy, P.~Kurasov, and D.~Mugnolo.
\newblock Surgery principles for the spectral analysis of quantum graphs.
\newblock {\em Trans. Amer. Math. Soc.}, 372(7):5153--5197, 2019.

\bibitem{BerKuc_incol12}
G.~Berkolaiko and P.~Kuchment.
\newblock Dependence of the spectrum of a quantum graph on vertex conditions
  and edge lengths.
\newblock In {\em Spectral Geometry}, volume~84 of {\em Proceedings of Symposia
  in Pure Mathematics}. American Math. Soc., 2012.
\newblock preprint {\tt arXiv:1008.0369}.

\bibitem{BerKuc_graphs}
G.~Berkolaiko and P.~Kuchment.
\newblock {\em Introduction to Quantum Graphs}, volume 186 of {\em Mathematical
  Surveys and Monographs}.
\newblock AMS, 2013.

\bibitem{BerLatSuk_am19}
G.~Berkolaiko, Y.~Latushkin, and S.~Sukhtaiev.
\newblock Limits of quantum graph operators with shrinking edges.
\newblock {\em Adv. Math.}, 352:632--669, 2019.

\bibitem{BolEnd_ahp09}
J.~Bolte and S.~Endres.
\newblock The trace formula for quantum graphs with general self adjoint
  boundary conditions.
\newblock {\em Ann. Henri Poincar\'e}, 10(1):189--223, 2009.

\bibitem{Bor_m21}
D.~I. Borisov.
\newblock Spectra of elliptic operators on quantum graphs with small edges.
\newblock {\em Mathematics}, 9(16):1874, 2021.

\bibitem{Bor_am22}
D.~I. Borisov.
\newblock Analyticity of resolvents of elliptic operators on quantum graphs
  with small edges.
\newblock {\em Adv. Math.}, 397:Paper No. 108125, 48, 2022.

\bibitem{Cac_s19}
C.~Cacciapuoti.
\newblock Scale invariant effective {H}amiltonian for a graph with a small
  compact core.
\newblock {\em Symmetry}, 11:359, 29, 2019.

\bibitem{CheExnTur_ap10}
T.~Cheon, P.~Exner, and O.~Turek.
\newblock Approximation of a general singular vertex coupling in quantum
  graphs.
\newblock {\em Ann. Physics}, 325(3):548--578, 2010.

\bibitem{DoKucOng_ems17}
N.~T. Do, P.~Kuchment, and B.~Ong.
\newblock On resonant spectral gaps in quantum graphs.
\newblock In {\em Functional analysis and operator theory for quantum physics},
  EMS Ser. Congr. Rep., pages 213--222. Eur. Math. Soc., Z\"{u}rich, 2017.

\bibitem{ExnGro_prep99}
P.~Exner and H.~Grosse.
\newblock Some properties of the one-dimensional generalized point interactions
  (a torso).
\newblock preprint {\tt arXiv:math-ph/9910029}, 1999.

\bibitem{ExnPos_jpa09}
P.~Exner and O.~Post.
\newblock Approximation of quantum graph vertex couplings by scaled
  {S}chr\"odinger operators on thin branched manifolds.
\newblock {\em J. Phys. A}, 42(41):415305, 22, 2009.

\bibitem{Gri_incol08}
D.~Grieser.
\newblock Thin tubes in mathematical physics, global analysis and spectral
  geometry.
\newblock In {\em Analysis on graphs and its applications}, volume~77 of {\em
  Proc. Sympos. Pure Math.}, pages 565--593. Amer. Math. Soc., Providence, RI,
  2008.

\bibitem{HornJohnson}
R.~A. Horn and C.~R. Johnson.
\newblock {\em Matrix analysis}.
\newblock Cambridge University Press, Cambridge, second edition, 2013.

\bibitem{KosSch_jpa99}
V.~Kostrykin and R.~Schrader.
\newblock Kirchhoff's rule for quantum wires.
\newblock {\em J. Phys. A}, 32(4):595--630, 1999.

\bibitem{KosSch_incol06}
V.~Kostrykin and R.~Schrader.
\newblock Laplacians on metric graphs: eigenvalues, resolvents and semigroups.
\newblock In {\em Quantum graphs and their applications}, volume 415 of {\em
  Contemp. Math.}, pages 201--225. Amer. Math. Soc., Providence, RI, 2006.

\bibitem{Kuc_wrm04}
P.~Kuchment.
\newblock Quantum graphs. {I}. {S}ome basic structures.
\newblock {\em Waves Random Media}, 14(1):S107--S128, 2004.
\newblock Special section on quantum graphs.

\bibitem{KucZen_incol03}
P.~Kuchment and H.~Zeng.
\newblock Asymptotics of spectra of {N}eumann {L}aplacians in thin domains.
\newblock In Y.~Karpeshina, G.~Stolz, R.~Weikard, and Y.~Zeng, editors, {\em
  Advances in differential equations and mathematical physics ({B}irmingham,
  {AL}, 2002)}, volume 327 of {\em Contemp. Math.}, pages 199--213. Amer. Math.
  Soc., Providence, RI, 2003.

\bibitem{LawTanChr_sr22}
T.~Lawrie, G.~Tanner, and D.~Chronopoulos.
\newblock A quantum graph approach to metamaterial design.
\newblock {\em Scientific Reports}, 12:18006, 2022.

\bibitem{Narasimhan_analysis_manifolds}
R.~Narasimhan.
\newblock {\em Analysis on real and complex manifolds}, volume~35 of {\em
  North-Holland Mathematical Library}.
\newblock North-Holland Publishing Co., Amsterdam, 1985.
\newblock Reprint of the 1973 edition.

\bibitem{Post_book12}
O.~Post.
\newblock {\em Spectral Analysis on Graph-like Spaces}, volume 2039 of {\em
  Lecture Notes in Mathematics}.
\newblock Springer Verlag, Berlin, 2012.

\bibitem{Sim_jfa77}
B.~Simon.
\newblock On the absorption of eigenvalues by continuous spectrum in regular
  perturbation problems.
\newblock {\em J. Functional Analysis}, 25(no. 4,):338--344, 1977.

\end{thebibliography}

\end{document}